\documentclass{article}
\usepackage{appendix}
\usepackage[a4paper]{geometry}
\usepackage[pdftex]{graphicx}
\usepackage[normalem]{ulem}
\usepackage{amssymb}
\usepackage[dvipsnames]{xcolor}
\definecolor{myblue}{RGB}{80,80,160}
\usepackage[utf8]{inputenc}
\usepackage[T1]{fontenc}
\usepackage[english]{babel}
\usepackage{array}
\usepackage{listings}
\usepackage{booktabs}
\usepackage{float}
\usepackage{amsmath}
\usepackage{cases}
\usepackage{mathtools}
\usepackage{mathdots}
\usepackage[overload]{empheq}
\usepackage[hidelinks,colorlinks=true,linkcolor=blue,citecolor=blue]{hyperref}

\makeatletter

\@addtoreset{equation}{section}
\makeatother
\usepackage{enumitem}
\usepackage{mathtools}
\usepackage{tikz}
\usetikzlibrary{shapes,patterns}
\usetikzlibrary{topaths,calc}
\usepackage{amsthm}
\newtheorem{thm}{Theorem}
\newtheorem{corollary}{Corollary}
\newtheorem{conjecture}{Conjecture}
\newtheorem{prop}{Proposition}
\newtheorem{lemme}{Lemma}
\newtheorem{obs}{Observation}
\usepackage[square,numbers]{natbib}
\usepackage{xurl}

\usepackage[mathscr]{euscript}

\title{\textbf{On Hypergraph Colorings and Completely Independent Spanning Trees in Chordal Graphs}}
\author{Mohammed Lalou, Nader Mbarek, Abdallah Skender, Olivier Togni\\
        Université Bourgogne Europe, LIB UR 7534, Dijon}
        
\begin{document}
\maketitle
\begin{abstract}
    In this paper, we study the existence problem of completely independent spanning trees (CIST) in chordal graphs through appropriate hypergraph representations and their panchromatic and bipanchromatic colorings. First, we disprove a conjecture stating an exact relationship between the panchromatic number, the bipanchromatic number, and the minimum number of unique colors in an optimal panchromatic coloring of a hypergraph.
    Then, by relating CIST to panchromatic and bipanchromatic colorings of the associated hypergraphs, we derive structural conditions for their existence in chordal graphs and specifically strictly chordal graphs. 
\end{abstract}

\section{Introduction}

Reliability is a fundamental requirement in the design of interconnected systems, where failures affecting devices or links can disrupt connectivity. Ensuring system robustness involves considering a primary path used during normal operation and an alternative one available when failures occur. This leads to considering several spanning trees that provide alternative ways of connecting the vertices, corresponding to the devices, of a graph representing the system. However, a single vertex or edge failure may affect several spanning trees simultaneously if they overlap. This motivates the study of families of spanning trees with strong independence properties, such as \textit{Completely Independent Spanning Trees} (\textit{CIST}) in which the paths between any pair of vertices are internally vertex-disjoint and edge-disjoint across the trees. In this context, CIST provide a strong structural redundancy and constitute an important graph-theoretic approach to fault tolerance. 

Hasunuma \cite{12} proved that the problem of determining whether a graph has two CIST is NP-complete and conjectured that any $2k$-connected graph contains $k$ CIST. However, this conjecture was disproved in \cite{4}, which showed that, for every $k \geq 2$, there exists a $k$-connected graph that does not admit two CIST. Due to the complexity of the problem, several sufficient conditions for the existence of $k$ CIST in a graph have been established \cite{5,40,22,48,73,38,37}. Further details can be found in \cite{82}. Alongside these general conditions, CIST have also been investigated in some subclasses of chordal graphs, such as split graphs \cite{29,hong_existence_2026,lalou_completely_2026} and $k$-trees  \cite{61}. In particular, Lalou et al. \cite{lalou_completely_2026} established a connection between panchromatic colorings of hypergraphs and the existence of completely independent spanning trees in split graphs. However, these studies are restricted to particular subclasses and do not provide a general framework for chordal graphs.

In this paper, we first disprove a conjecture concerning an exact relationship between the panchromatic number, the bipanchromatic number, and the minimum number of unique colors in an optimal panchromatic coloring. Then, we extend the method of Lalou et al. \cite{lalou_completely_2026} from split graphs to chordal graphs by developing a closely related hypergraph-based correspondence between the existence of CIST and panchromatic and bipanchromatic colorings of associated hypergraphs. 

The rest of this paper is organized as follows. Section \ref{sec2} presents the preliminaries used throughout this paper, namely the panchromatic and bipanchromatic coloring of hypergraphs, as well as the definition of the clique hypergraph corresponding to a chordal graph and its associated separator hypergraph to a chordal graph. Section \ref{sec3} constructs an infinite family of hypergraphs that disproves the conjectured relationship between the panchromatic and bipanchromatic numbers of a hypergraph. Finally, Section \ref{sec4} presents general results on the number of CIST in chordal graphs according to the panchromatic and bipanchromatic colorings of the associated hypergraphs.

\section{Preliminaries} \label{sec2}
In this paper, graphs, representing networked systems, are assumed to be finite, simple, and undirected. For a graph $G$, we denote its vertex and edge sets by $V(G)$ and $E(G)$, respectively. The neighborhood of a vertex $x$, denoted $N_G(x)$, is the set of all vertices adjacent to $x$, and the {\em degree} of $x$ in $G$ is $d_G(x) = |N_G(x)|$. Let $x,y$ be two vertices of $G$, a {\em $(x,y)$-path} is a sequence of distinct vertices starting with $x$ and ending with $y$, where each consecutive pair is an edge of $E(G)$. $G$ is said to be {\em connected} if there exists a path between any pair of its vertices; it is {\em $k$-connected} if it remains connected even after deleting any set of $k-1$ vertices. An {\em acyclic} graph is a graph that contains no cycles. A {\em tree} $T$ is an acyclic connected graph. A {\em spanning tree} $T$ of a graph $G$ is a tree such that $V(T) = V(G)$ and $E(T) \subseteq E(G)$. A vertex $x \in V(T)$ is an {\em internal} vertex if $d_T(x) \geq 2$, otherwise, it is a {\em leaf}. If an internal vertex $x$ of a tree $T$ is adjacent to a leaf $y$ of $T$, then $x$ is said to {\em cover} $y$ in $T$.

Let $G$ be a graph and $P_1,P_2$ two $(x,y)$-paths. $P_1$ and $P_2$ are {\em internally-disjoint} if they do not have an internal vertex in common, and they are {\em openly-disjoint} if they are both edge-disjoint and internally-disjoint. Let $k\geq 2$ be an integer and $T_1, \dots , T_k$ be spanning trees of a graph $G$. If for any pair of vertices $x,y \in V(G)$, the $(x,y)$-paths in $T_1, \dots , T_k$ are pairwise openly-disjoint, then $T_1, \dots , T_k$ are called CIST of $G$. CIST were introduced by Hasunuma \cite{2}, who provided a characterization of them:
\begin{thm}[\cite{2}]
    Let $T_1, \dots, T_k$ be spanning trees in a graph $G$. Then, $T_1, \dots , T_k$ are  completely independent if and only if $T_1, \dots , T_k$ are edge-disjoint and for any vertex $x \in V(G)$ and $1 \leq i \leq k$, there is at most one spanning tree $T_i$ such that $d_{T_i}(x) > 1.$
    \label{thm00}
\end{thm}

A {\em chordal graph}  $G = (V,E)$ is a graph in which all cycles of at least four vertices have a \textit{chord}, \textit{i.e.}, an edge joining two nonconsecutive vertices of the cycle. Let $S$ be a subset of $V(G)$, $S$ is a \textit{minimal vertex separator} of $G$ if there exist two nonadjacent vertices $u,v\in V(G)\setminus S$ such that $u$ and $v$ are in different connected components of $G-S$, and no subset $S' \subsetneq S$ separates $u$ and $v$. Dirac \cite{dirac_rigid_1961} proved that $G$ is chordal if and only if every minimal vertex separator of $G$ is a clique in $G$. Also, by \cite{buneman_characterisation_1974}, every minimal separator of a chordal graph is properly (\textit{i.e.}, strictly) contained in at least two distinct maximal cliques. Furthermore, Blair and Peyton \cite{blair_introduction_1993} showed that a connected graph $G$ is chordal if and only if there exists a tree $T$ called a \emph{clique tree} of $G$ whose vertices are the maximal cliques of $G$ which satisfies the \textit{clique-intersection property}: for any two maximal cliques $K$ and $K'$, the intersection $K \cap K'$ is contained in every maximal clique lying on the unique path between $K$ and $K'$ in $T$. 

A hypergraph $H = (V, \mathcal{E}) $ is a generalization of a graph with $V$ being the set of vertices and $\mathcal{E}$ the set of hyperedges $\mathcal{E} \subset \mathcal{P}(V)$, in which a {\em hyperedge} $e \in \mathcal{E}$ can join any number of vertices. $H$ is {\em $k$-uniform} if every hyperedge contains exactly $k$ vertices. For a chordal graph $G$, let $\mathcal{K}(G)$ and $\mathcal{S}(G)$ denote the families of maximal cliques and minimal vertex separators of $G$, respectively. Viewing these cliques and separators as hyperedges yields the following hypergraph representations. The \textit{clique hypergraph} of $G$, denoted by $H_C(G)$, is the hypergraph defined by $V(H_C(G))=V(G)$, and $\mathcal{E}(H_C(G))=\mathcal{K}(G)$. The \textit{separator hypergraph} associated with $H_C(G)$, denoted by $H_S(G)$, is the hypergraph defined by $V(H_S(G)) = \bigcup_{S \in \mathcal{S}(G)} S$ and $\mathcal{E}(H_S(G)) = \mathcal{S}(G)$. Since two distinct vertices are adjacent in $G$ if and only if they belong to a common hyperedge in $H_C(G)$, the mapping $G \mapsto H_C(G)$ establishes a bijection between chordal graphs and their clique hypergraphs. Thus, $G$ uniquely determines $H_C(G)$, and conversely, $H_C(G)$ uniquely determines $G$ and hence its family of minimal vertex separators. Consequently, $H_C(G)$ also uniquely determines the associated separator hypergraph $H_S(G)$. However, the converse does not hold, since the minimal separators alone do not determine all maximal cliques of $G$. Hence, different clique hypergraphs may have the same associated separator hypergraph. Figure~\ref{sep2cli} shows a separator hypergraph $H_S$ associated with two distinct clique hypergraphs $H_C^1$ and $H_C^2$.

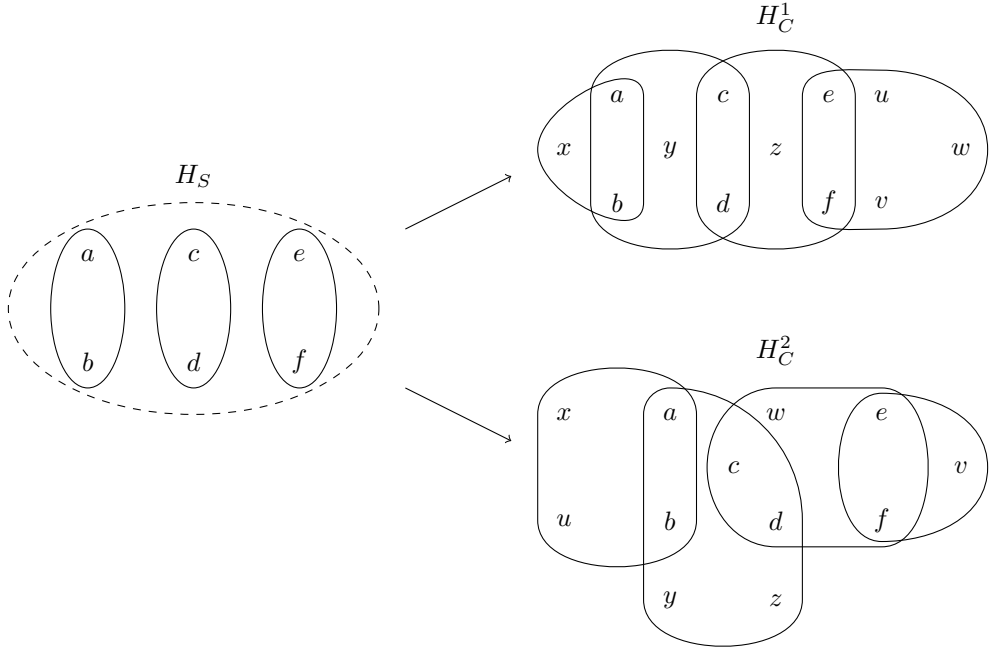
\begin{figure}[ht]
\centering
\begin{tikzpicture}[scale = 0.7]
        \node[shape=circle] (1) at (6,1) {$x$};        
        \node[shape=circle] (2) at (7,2) {$a$};
        \node[shape=circle] (3) at (7,0) {$b$};
        \node[shape=circle] (4) at (9,2) {$c$};        
        \node[shape=circle] (5) at (9,0) {$d$};
        \node[shape=circle] (6) at (8,1) {$y$};
        \node[shape=circle] (7) at (10,1) {$z$};        
        \node[shape=circle] (8) at (11,2) {$e$};
        \node[shape=circle] (9) at (11,0) {$f$};
        \node[shape=circle] (10) at (12,2) {$u$};
        \node[shape=circle] (11) at (12,0) {$v$};        
        \node[shape=circle] (12) at (13.5,1) {$w$};
        
        \begin{scope}[fill opacity=0]
        \filldraw[fill=yellow!70] ($(1)+(-0.5,0)$)
            to[out=90,in=90] ($(2) + (0.5,0)$)
            to[out=-90,in=90] ($(3) + (0.5,0)$)
            to[out=-90,in=-90] ($(1) + (-0.5,0)$) ;
        \end{scope}

        \begin{scope}[fill opacity=0]
        \filldraw[fill=yellow!70] ($(2)+(-0.5,0)$)
            to[out=90,in=90] ($(4) + (0.5,0)$)
            to[out=-90,in=90] ($(5) + (0.5,0)$)
            to[out=-90,in=-90] ($(3) + (-0.5,0)$)
           to[out=90,in=-90] ($(2) + (-0.5,0)$) ;
        \end{scope}

        \begin{scope}[fill opacity=0]
        \filldraw[fill=yellow!70] ($(4)+(-0.5,0)$)
            to[out=90,in=90] ($(8) + (0.5,0)$)
            to[out=-90,in=90] ($(9) + (0.5,0)$)
            to[out=-90,in=-90] ($(5) + (-0.5,0)$)
           to[out=90,in=-90] ($(4) + (-0.5,0)$) ;
        \end{scope}

        \begin{scope}[fill opacity=0]
        \filldraw[fill=yellow!70] ($(8)+(-0.5,0)$)
            to[out=90,in=180] ($(10) + (0,0.5)$)
            to[out=0,in=90] ($(12) + (0.5,0)$)
            to[out=-90,in=0] ($(11) + (0,-0.5)$)
            to[out=180,in=-90] ($(9) + (-0.5,0)$)
           to[out=90,in=-90] ($(8) + (-0.5,0)$) ;
        \end{scope}

        \node[shape=circle] (01) at (6,-4) {$x$};   
        \node[shape=circle] (02) at (6,-6) {$u$};
        \node[shape=circle] (03) at (8,-4) {$a$};
        \node[shape=circle] (04) at (8,-6) {$b$};
        \node[shape=circle] (05) at (9.2,-5) {$c$};        
        \node[shape=circle] (06) at (8,-7.5) {$y$};
        \node[shape=circle] (07) at (10,-6) {$d$};
        \node[shape=circle] (08) at (10,-7.5) {$z$};        
        \node[shape=circle] (09) at (10,-4) {$w$};
        \node[shape=circle] (010) at (12,-4) {$e$};
        \node[shape=circle] (011) at (12,-6) {$f$};        
        \node[shape=circle] (012) at (13.5,-5) {$v$};

        \begin{scope}[fill opacity=0]
        \filldraw[fill=yellow!70] ($(03)+(0,0.5)$)
            to[out=0,in=90] ($(07) + (0.5,0)$)
            to[out=-90,in=90] ($(08) + (0.5,0)$)
            to[out=-90,in=-90] ($(06) + (-0.5,0)$)
            to[out=90,in=-90] ($(03) + (-0.5,0)$)
           to[out=90,in=180] ($(03) + (0,0.5)$) ;
        \end{scope}
        
        \begin{scope}[fill opacity=0]
        \filldraw[fill=yellow!70] ($(01)+(-0.5,0)$)
            to[out=90,in=90] ($(03) + (0.5,0)$)
            to[out=-90,in=90] ($(04) + (0.5,0)$)
            to[out=-90,in=-90] ($(02) + (-0.5,0)$)
           to[out=90,in=-90] ($(01) + (-0.5,0)$) ;
        \end{scope}
        \begin{scope}[fill opacity=0]
        \filldraw[fill=yellow!70] ($(09)+(0,0.5)$)
            to[out=0,in=180] ($(010) + (0,0.5)$)
            to[out=0,in=0] ($(011) + (0,-0.5)$)
            to[out=-90,in=0] ($(011) + (0,-0.5)$)
            to[out=180,in=0] ($(07) + (0,-0.5)$)
            to[out=180,in=-90] ($(05) + (-0.5,0)$)
           to[out=90,in=180] ($(09) + (0,0.5)$) ;
        \end{scope}
        \begin{scope}[fill opacity=0]
        \filldraw[fill=yellow!70] ($(010)+(0,0.4)$)
            to[out=0,in=90] ($(012) + (0.5,0)$)
            to[out=-90,in=0] ($(011) + (0,-0.4)$)
            to[out=180,in=180] ($(010) + (0,0.4)$) ;
        \end{scope}

        \node[shape=circle] (13) at (-3,-1) {$a$};
        \node[shape=circle] (14) at (-3,-3) {$b$};
        \node[shape=circle] (15) at (-1,-1) {$c$};        
        \node[shape=circle] (17) at (-1,-3) {$d$};
        \node[shape=circle] (110) at (1,-1) {$e$};
        \node[shape=circle] (111) at (1,-3) {$f$};  

\draw[] (-3,-2) ellipse (0.7cm and 1.5cm);
\draw[] (-1,-2) ellipse (0.7cm and 1.5cm);
\draw[] (1,-2) ellipse (0.7cm and 1.5cm);
\draw[dashed] (-1,-2) ellipse (3.5cm and 2cm);
\draw[->] (3,-0.5) -- (5,0.5);
\draw[->] (3,-3.5) -- (5,-4.5);

        \node (40) at (-1,0.5) {$H_S$};        
        \node (41) at (10,3.5) {$H_C^1$};        
        \node (42) at (10,-2.8) {$H_C^2$};        
\end{tikzpicture}
\caption{A separator hypergraph $H_S$ associated with two distinct clique hypergraphs $H_C^1$ and $H_C^2$}
\label{sep2cli}
\end{figure}

\begin{obs}
Let $G=(D\cup I, E)$ be a split graph, where $D$ is a maximal clique and $I$ is an independent set. Following \cite{lalou_completely_2026}, its corresponding hypergraph is defined as $H(G)=(D,\mathcal{E})$, such that $\mathcal{E}=\{N_G(x):x\in I\}$. Every minimal vertex separator of $G$ is the neighborhood of some vertex in $I$. Consequently, $H_S(G)$ is a subhypergraph of $H(G)$, but equality does not always hold. Therefore, the separator hypergraph representation defined here can be considered as an extension of the approach from split graphs to chordal graphs.
\end{obs}

We next introduce the hypergraph coloring notions used to study the existence of CIST. A $k$-coloring of a hypergraph $H$ is an assignment of $k$ distinct colors to the vertices of $H$ so that no hyperedge is monochromatic. A {\em panchromatic} $k$-coloring of $H$ is a coloring of the vertices of $H$ such that each hyperedge contains at least one vertex of each color. A color that appears once in the hypergraph is called a {\em unique} color. The problem of deciding whether a hypergraph is panchromatically $2$-colorable is NP-Complete \cite{garey}, and by extension for $k\geq 3$. Moreover, if a panchromatic $k$-coloring $\varphi$ of a hypergraph $H$ is such that each of the $k$ colors appears at least twice in $H$, then $\varphi$ is called {\em bipanchromatic} \cite{lalou_completely_2026}. 
The panchromatic (resp. bipanchromatic) number, denoted $\chi_p(H)$ (resp. $\chi^2_p(H)$), is the maximum $k$ such that $H$ admits a panchromatic (resp. bipanchromatic) $k$-coloring. 

\section{Counterexample}
\label{sec3}
Lalou et al. \cite{lalou_completely_2026} asked whether a bipanchromatic $\chi_p^2(H)$-coloring of $H$ can be obtained from a panchromatic $\chi_p(H)$-coloring that minimizes the number of vertices of unique color $\alpha(H)$, by repeatedly pairing the unique colors until none remain. Based on their calculations, they proposed the following conjecture:
\begin{conjecture}[\cite{lalou_completely_2026}]
    Every hypergraph $H$ satisfies $\chi^2_p (H)= \chi_p (H) -  \left\lceil \frac{\alpha(H)}{2} \right\rceil$.
    \label{conj}
\end{conjecture}

If true, this conjecture would be particularly useful, as it would restrict the maximum number of CIST in a split graph to one of the following two values: $\chi_p (H(G)) -  \left\lceil \frac{\alpha}{2} \right \rceil$ and $\chi_p (H(G)) -  \left\lceil \frac{\alpha}{2}\right \rceil + 1$, where $\alpha=\alpha(H(G))$ \cite[Proposition~3]{lalou_completely_2026}. However, in what follows, we prove that this conjecture does not hold in general by constructing an infinite family of hypergraphs that violates it:
\begin{prop}
    There exists an infinite number of hypergraphs that do not satisfy Conjecture \ref{conj}.
\end{prop}
\begin{proof}
    Let $d \geq1$ and $q \geq 2d$ two integers. Consider the four disjoint sets of vertices $U = \{ u_1,u_2,\dots,u_{6d} \}$, $A = \{ a_1,a_2,\dots,a_q \}$, $B = \{ b_1,b_2,\dots,b_q \}$ and $C = \{ c_1,c_2,\dots,c_q \}$. We define the hypergraph $\mathcal{H}_{d,q}$ such that $V(\mathcal{H}_{d,q}) = U \cup A \cup B \cup C$ and $\mathcal{E}(\mathcal{H}_{d,q}) = \{ U \cup A, U \cup B, U \cup C \}$ as illustrated in Figure~\ref{t}. Clearly, $\mathcal{H}_{d,q}$ is $(6d +q)$-uniform, and hence $\chi_p \leq 6d+q$. We next show that $\mathcal{H}_{d,q}$ does not satisfy the conjecture: 
    $$\chi_p(\mathcal{H}_{d,q}) = 6d+q \ , \
    \alpha(\mathcal{H}_{d,q}) = 6d \ , \  \mbox{ and } \ 
    \chi_p^2(\mathcal{H}_{d,q}) = 4d+q > \chi_p - \left\lceil \frac{\alpha}{2} \right\rceil = 3d+q .
    $$
\begin{sloppypar}
We first construct a panchromatic $(6d+q)$-coloring of $\mathcal{H}_{d,q}$. Let $P =\{p_1,p_2,\dots,p_{6d+q} \}$ be a set of colors. For each $i\in \{1,\ldots,6d \}$, assign color $p_i$ to $u_i$. Moreover, for each $j\in \{1,\ldots,q\}$, assign color $p_{6d+j}$ to each of the three vertices $a_j$, $b_j$, and $c_j$. Each hyperedge $U\cup A$, $U\cup B$, and $U\cup C$ therefore contains every color in $P$. Hence, $\chi_p(\mathcal{H}_{d,q})\geq 6d+q$. Together with the bound $\chi_p(\mathcal{H}_{d,q})\leq 6d+q$, this gives $\chi_p(\mathcal{H}_{d,q})=6d+q$.
\end{sloppypar}
Now consider any panchromatic $\chi_p(\mathcal{H}_{d,q})$-coloring. Since every hyperedge contains exactly $\chi_p(\mathcal{H}_{d,q})$ vertices, all vertices within each hyperedge must have distinct colors. In particular, the vertices of $U$, which belong to all three hyperedges, must receive pairwise distinct colors. Furthermore, no vertex in $A\cup B\cup C$ can receive a color assigned to a vertex of $U$, since it will duplicate that color in one of the hyperedges and force another color to be absent. Therefore, each color assigned to a vertex of $U$ appears exactly once in the entire hypergraph. Thus, $\alpha=6d$.

Next, we show that $\chi_p^2(\mathcal{H}_{d,q})=4d+q$. We first construct a bipanchromatic $4d+q$-coloring of $\mathcal{H}_{d,q}$. Let $U$ be partitioned equally into three sets $U_A$, $U_B$, and $U_C$, each of cardinality $2d$. Let $P'= \{p'_1,p'_2,\dots,p'_{4d+q} \}$ be a set of colors. Pair each vertex of $U_A$ with a distinct vertex of $A$, each vertex of $U_B$ with a distinct vertex of $B$, and each vertex of $U_C$ with a distinct vertex of $C$; then, assign a distinct color to the two vertices of each pair. This gives pairs colored with $6d$ distinct colors. Next, $q-2d$ vertices remain uncolored in each of $A$, $B$, and $C$. Group the remaining vertices of these three sets into $q-2d$ triples, each consisting of one vertex from $A$, one from $B$, and one from $C$, and assign a distinct new color to all three vertices of each triple. Therefore, the number of colors used is $6d + (q-2d) = 4d+q$.

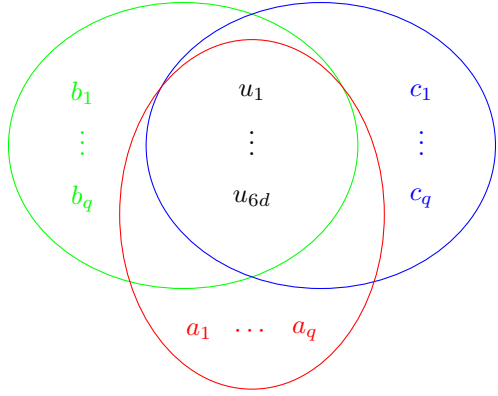
\begin{figure}[ht]
    \begin{center}
    \begin{tikzpicture}[scale = 0.7]
        \node[shape=circle] (1) at (1,0) {$u_1$};
        \node[shape= circle] (40) at (1,-0.8) {$\vdots$};         
        \node[shape=circle] (7) at (1,-2) {$u_{6d}$};

        \node[shape=circle,red] (8) at (0,-4.5) {$a_1$};
        \node[shape=circle,red] (100) at (1,-4.5) {$\cdots$};
        \node[shape=circle,red] (11) at (2,-4.5) {$a_q$};

        \node[shape=circle,green] (12) at (-2.2,0) {$b_1$};
        \node[shape=circle,green] (130) at (-2.2,-0.8) {$\vdots$};
        \node[shape=circle,green] (14) at (-2.2,-2) {$b_q$};

        \node[shape=circle,blue] (15) at (4.2,0) {$c_1$};
        \node[shape=circle,blue] (160) at (4.2,-0.8) {$\vdots$};
        \node[shape=circle,blue] (17) at (4.2,-2) {$c_q$};

\draw[green] (-0.3,-1) ellipse (3.3cm and 2.7cm);
\draw[blue] (2.3,-1) ellipse (3.3cm and 2.7cm);
\draw[red] (1,-2.3) ellipse (2.5cm and 3.3cm);
    \end{tikzpicture}
        \caption{Illustration of the hypergraph $\mathcal{H}_{d,q}$}
        \label{t}
    \end{center}
\end{figure}

Moreover, this coloring is also an optimal bipanchromatic coloring. Consider an arbitrary bipanchromatic coloring $\phi$ of $\mathcal{H}_{d,q}$, and partition its color classes into two sets $X\cup Y$, where $X$ denotes the set of color classes containing at least one vertex of $U$, and $Y$ denotes the set of color classes containing no vertex of $U$. Since $|U|=6d$, we have $|X| \leq 6d$.
Moreover, every color class in $Y$ must intersect each of the three hyperedges. And since it contains no vertex of $U$, it must contain at least one vertex from each of $A$, $B$, and $C$. Hence, $|Y| \leq q$. In order to maximize the total number of colors, each color class in $X$ should contain exactly one vertex of $U$ and one vertex of $A \cup B \cup C$. Therefore, the classes in $X$ use $6d$ vertices of $A\cup B\cup C$. The remaining $3q-6d$ vertices can then form $\frac{3q-6d}{3}$ color classes in $Y$. Therefore, $|\phi| = 6d+\frac{3q-6}{3} = 4d+q$. Consequently, $\chi_p^2(\mathcal{H}_{d,q}) = 4d+q$. 
\end{proof}
This construction also gives an infinite family of hypergraphs whose bipanchromatic number exceeds the value predicted by the conjecture by exactly $d$. For example, for $d=3$ and $q = 7$,  $\chi_p^2(\mathcal{H}_{3,7})= 19$; however, the conjecture would predict $\chi_p^2(\mathcal{H}_{3,7}) = \chi_p(\mathcal{H}_{3,7})- \frac{\alpha(\mathcal{H}_{3,7})}{2} = 16$.

Since Conjecture~\ref{conj} does not hold in general, it is preferable to work directly with bipanchromatic colorings when deriving existence conditions for CIST. Indeed, deriving a bipanchromatic coloring by pairing colors from a panchromatic coloring may yield fewer colors than optimizing directly over the bipanchromatic coloring. Note that the number of colors in a bipanchromatic coloring influences the number of CIST that can be constructed, which we aim to maximize. Moreover, in \cite{lalou_completely_2026}, the panchromatic and bipanchromatic colorings considered in this conjecture were used to study CIST in split graphs. We therefore extend this approach to chordal graphs, for which it is particularly well-suited because bipanchromatic colorings of their separator hypergraphs provide structural information that can be exploited in the construction of CIST, as we will see in the next section.

\section{Hypergraph-Based Conditions for CIST in Chordal Graphs}
In this section, using the clique hypergraph $H_C(G)$ and its associated separator hypergraph $H_S(G)$, we establish relationships between panchromatic and bipanchromatic colorings and the existence of CIST in chordal graphs. Moreover, since the optimal number of CIST in complete graphs has already been determined \cite{3}, we consider in this research work chordal graphs that are not complete graphs.

\begin{thm}
Let $G$ be a connected chordal graph. If $G$ admits $k$ CIST, then $H_C(G)$ and $H_S(G)$ both admit a panchromatic $k$-coloring.
\end{thm}
\begin{proof}
Let $T_1,\dots,T_k$ be $k$ CIST of $G$. By Theorem~\ref{thm00}, every vertex of $G$ is internal in at most one tree among $T_1,\dots,T_k$. Let $\varphi$ be the $k$-coloring of $H_S(G)$ defined as follows. If a vertex $x$ is internal in $T_i$, then color $i$ is assigned to $x$. The vertices that are leaves in all the trees are assigned an arbitrary color. 

Let $S\in E(H_S(G))$ be a minimal vertex separator of $G$. By definition, there exist two vertices $u,v\in V(G)\setminus S$ that are in distinct connected components of $G-S$. For every $i \in {1,\dots,k}$, consider the unique path $P_i$ that connects $u$ and $v$ in $T_i$. Since $S$ separates $u$ and $v$ in $G$, the path $P_i$ must contain at least one vertex $x_i\in S$; therefore, $x_i$ is an internal vertex of $P_i$ and, consequently, an internal vertex of $T_i$. By the definition of $\varphi$, $x_i$ is assigned color $i$. Thus, $S$ contains at least one vertex of each color $i\in{1,\dots,k}$. It follows that every hyperedge of $H_S(G)$ contains all $k$ colors. Consequently, $\varphi$ is a panchromatic $k$-coloring of $H_S(G)$.

Also, every maximal clique of $G$ contains a minimal vertex separator which is the intersection with one of its neighboring cliques. It follows that, when extending $\varphi$ to $H_C(G)$ by assigning an arbitrary color to every vertex not belonging to $V(H_S(G))$, every hyperedge of $H_C(G)$ will contain all $k$ colors. Therefore, the extended $k$-coloring of $H_C(G)$ is panchromatic.
\end{proof}

\begin{lemme}
Let $G$ be a connected chordal graph, and let $\varphi$ be a panchromatic $k$-coloring of $H_S(G)$, with color classes $D_1,\dots,D_k$. Then each $D_i$ is a connected dominating set of $G$.
\label{lemme0}
\end{lemme}
\begin{proof}
Let $T$ be a clique tree of $G$ such that the intersections of its endpoint cliques label its edges. These intersections are minimal vertex separators of $G$. For $i \in {1,\dots,k}$. Let $x,y\in D_i$, and let $S_x$ and $S_y$ be minimal separators containing $x$ and $y$, respectively. Over the path containing the edge labeled by $S_x$ to the one labeled by $S_y$ in $T$, every edge label is a minimal separator; therefore, it contains a vertex of color $i$. Moreover, the vertices of color $i$ that are chosen from two consecutive edge labels belong to a common maximal clique and are thus adjacent. Hence, these vertices form a path that connects $x$ to $y$, and $D_i$ is then connected. 

Let $v \in V(G)\setminus D_i$, and let $K$ be a maximal clique containing $v$. $K$ contains a minimal separator $S$. Since $\varphi$ is panchromatic, $S$ contains a vertex $x\in D_i$. And since $v,x\in K$, they are adjacent. Therefore, every vertex outside of $D_i$ has a neighbor in $D_i$. Consequently, $D_i$ is a connected dominating set of $G$.
\end{proof}

Thus, each color class can serve as the set of internal vertices of a spanning tree. Moreover, to ensure that the vertices belonging to the other color classes can be attached without creating common edges between the resulting trees, we introduce the following property.

Let $\varphi$ be a panchromatic $k$-coloring of $H_S(G)$, with color classes $D_1,\ldots,D_k$. We say that $\varphi$ is \textit{pairwise separated} if, for every pair $1\leq i<j\leq k$, there exist two edge-disjoint sets
$F_{ij}^{i},F_{ij}^{j}\subseteq \{ (x,y) \in E(G) | x \in D_i, y \in D_j \}$ such that every vertex of $D_j$ is incident with an edge of $F_{ij}^{i}$, and every vertex of $D_i$ is incident with an edge of $F_{ij}^{j}$.

\begin{thm}
Let $G$ be a connected chordal graph. If $H_S(G)$ admits a pairwise separated bipanchromatic $k$-coloring, then $G$ has $k$ CIST.
\label{thm01}
\end{thm}

\begin{proof}
Let $\varphi$ be a bipanchromatic $k$-coloring with color classes $D_1,\dots,D_k$. By Lemma~\ref{lemme0}, each $D_i$ is a connected dominating set of $G$. It follows that, for every $i$, the subgraph $G[D_i]$ contains a spanning tree $T'_i$ that will be used to construct a spanning tree $T_i$ in $G$. For every $j \neq i$ and every vertex $v\in D_j$, add exactly one edge of $F_{ij}^{i}$ incident with $v$, to $T_i$. Moreover, for every vertex $v\in V(G)\setminus\bigcup_{t=1}^{k}D_t$, take one neighbor $x_i \in D_i$ (this neighbor exists since $D_i$ is dominating), then add the edge $(v,x_i)$ to $T_i$. With this construction, every vertex outside $D_i$ is connected to $T_i$ by one edge. Consequently, $T_i$ is a spanning tree of $G$. Since the color classes are pairwise disjoint, the trees $T_1,\dots,T_k$ are internally disjoint.

These trees are also edge-disjoint. The edges of $T'_i$ have both extremities in $D_i$; however, every edge added to construct $T_i$ has exactly one extremity in $D_i$, which means that these edges do not belong to $T'_i$.  Since $\varphi$ is pairwise separated, for every pair $i \neq j$, the edges from $F_{ij}^{i}$ used to connect $D_j$ to $D_i$ in $T_i$ and the ones from $F_{ij}^{j}$ used to attach $D_i$ to $D_j$ in $T_j$ are edge-disjoint. Hence, $T_1,\dots,T_k$ are $k$ CIST of $G$.
\end{proof}
Figure~\ref{expl} illustrates a chordal graph admitting two CIST and a pairwise separated bipanchromatic $2$-coloring of its associated separator hypergraph.
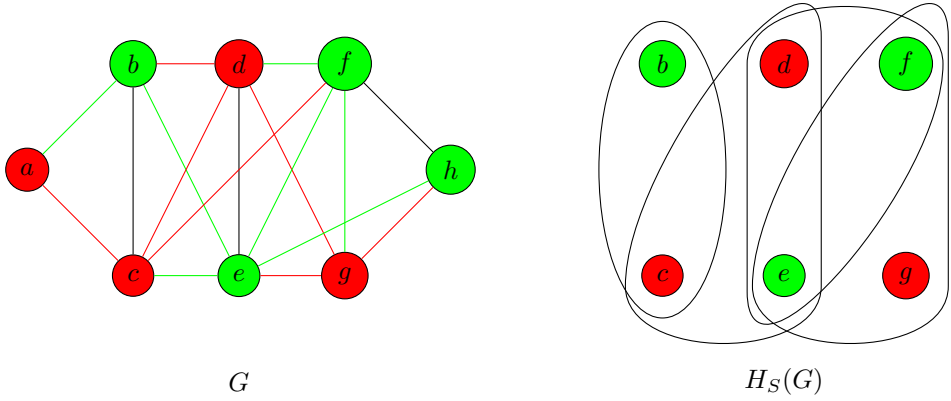
\begin{figure}[ht]
\centering
\begin{tikzpicture}[scale = 0.7]        
        \node[shape=circle,draw=black,fill=red] (2) at (0,0) {$a$};
        \node[shape=circle,draw=black,fill=green] (3) at (2,2) {$b$};
        \node[shape=circle,draw=black,fill=red] (4) at (2,-2) {$c$}; 
        \node[shape=circle,draw=black,fill=red] (5) at (4,2) {$d$};
        \node[shape=circle,draw=black,fill=green] (6) at (4,-2) {$e$};
        \node[shape=circle,draw=black,fill=green] (7) at (6,2) {$f$}; 
        \node[shape=circle,draw=black,fill=red] (8) at (6,-2) {$g$};
        \node[shape=circle,draw=black,fill=green] (9) at (8,0) {$h$};

    \draw[-,green] (2) -- (3);
    \draw[-,red] (2) -- (4);
    \draw[-] (4) -- (3);
    \draw[-,red] (5) -- (3);
    \draw[-,green] (4) -- (6);
    \draw[-,green] (6) -- (3);
    \draw[-,red] (5) -- (4);
    \draw[-] (5) -- (6);
    \draw[-,green] (5) -- (7);
    \draw[-,red] (8) -- (6);
    \draw[-,green] (7) -- (8);
    \draw[-,red] (4) -- (7);
    \draw[-] (7) -- (9);
    \draw[-,red] (5) -- (8);
    \draw[-,green] (6) -- (7);
    \draw[-,green] (6) -- (9);
    \draw[-,red] (8) -- (9);

    \node[shape=circle,draw=black,fill = green] (04) at (12,2) {$b$};
    \node[shape=circle,draw=black,fill=red] (05) at (12,-2) {$c$};    
    \node[shape=circle,draw=black,fill=red] (07) at (14.3,2) {$d$};
    \node[shape=circle,draw=black,fill=green] (08) at (14.3,-2) {$e$};
    \node[shape=circle,draw=black,fill=green] (011) at (16.6,2) {$f$}; 
    \node[shape=circle,draw=black,fill=red] (012) at (16.6,-2) {$g$};

    \draw[] (12,0) ellipse (1.2cm and 2.8cm);

        \begin{scope}[fill opacity=0]
        \filldraw[fill=yellow!70] ($(05)+(-0.7,-0.2)$)
            to[out=90,in=90] ($(07) + (0.7,0.5)$)
            to[out=-90,in=90] ($(08) + (0.7,-0.2)$)
            to[out=-90,in=-90] ($(05) + (-0.7,-0.2)$); 
        \end{scope}

        \begin{scope}[fill opacity=0]
        \filldraw[fill=yellow!70] ($(07)+(-0.7,0)$)
            to[out=90,in=90] ($(011) + (0.7,0)$)
            to[out=-90,in=-90] ($(08) + (-0.7,-0.3)$)
            to[out=90,in=-90] ($(07) + (-0.7,0)$); 
        \end{scope}

        \begin{scope}[fill opacity=0]
        \filldraw[fill=yellow!70] ($(08)+(-0.6,-0.2)$)
            to[out=90,in=90] ($(011) + (0.8,0.5)$)
            to[out=-90,in=90] ($(012) + (0.8,-0.2)$)
            to[out=-90,in=-90] ($(08) + (-0.6,-0.2)$); 
        \end{scope}
        
        \node (41) at (4,-4) {$G$};        
        \node (42) at (14.3,-4) {$H_S(G)$};        
\end{tikzpicture}
\caption{A chordal graph $G$ with two CIST whose separator hypergraph $H_S(G)$ admits a pairwise separated bipanchromatic $2$-coloring}
\label{expl}
\end{figure}

In split graphs, any bipanchromatic coloring is pairwise separated. Indeed, all vertices of the associated hypergraph belong to the main clique of the split graph and are therefore pairwise adjacent. Moreover, the bipanchromaticity guarantees that every color class contains at least two vertices, allowing the edges required in both directions between any two color classes to be chosen while guaranteeing disjunction. Consequently, in split graphs, a bipanchromatic $k$-coloring alone is sufficient to guarantee the existence of $k$ CIST. A class of chordal graphs that may satisfy the same property as split graphs is the class of strictly chordal graphs, that is, chordal graphs in which all minimal vertex separators are disjoint \cite{markenzon_new_2015}.

\begin{corollary} \label{coro1}
Let G be a connected strictly chordal graph. If the separator hypergraph $H_S(G)$ admits a bipanchromatic $k$-coloring, then $G$ admits $k$ CIST.
\end{corollary}

\begin{proof}
Since the coloring is bipanchromatic, each minimal separator contains every color. Hence, for every pair of colors $i\neq j$, each vertex of $D_i$ is adjacent to a vertex of $D_j$, and vice versa. These edges in both directions can be chosen distinctly because along a clique tree, any two consecutive minimal separators are contained in a common maximal clique. Thus, the coloring is pairwise separated, and the result follows from Theorem~\ref{thm01}.
\end{proof}

\begin{corollary}
Let G be a connected strictly chordal graph and $H_S(G)$ its separator hypergraph such that $\mathcal{E}(H_S(G)) = \{ S_1,S_2, \dots, S_{\ell} \}$. If $\ell = 1$, then $G$ admits $\lfloor \frac{|S_1|}{2} \rfloor$ CIST. Otherwise, $G$ admits $S_{\min}$ CIST, with $S_{\min} = \min_{1 \leq i \leq \ell} |S_i|$.
\end{corollary}

\begin{proof}
Suppose that $\ell = 1$. Let $\{ u_1,u_2, \dots, u_{|S_1|} \}$ be the vertices of $S_1$ and let $s = \lfloor \frac{S_1}{2} \rfloor$. Let $P = \{ p_1,p_2,\dots, p_s \}$ be a set of colors. For each $i \in \{ 1,2,\dots, s \}$, assign the color $i$ to the vertices $u_{2i}$ and $u_{2i-1}$. Each color of $P$ is present twice in the hyperedge; therefore, $H_S(G)$ admits a bipanchromatic $s$-coloring. From Corollary~\ref{coro1}, $G$ admits $s$ CIST.

Now suppose that $\ell \geq 2$. For each $i \in \{ 1,2,\dots, m \}$, let $\{ v_i^1,v_i^2, \dots, v_i^{|S_i|} \}$ be the set of vertices of $S_i$. Let $P = \{ p_1,p_2,\dots, p_{|S_{\min}|} \}$ be a set of colors. For each $i$, assign the color $p_j$ to the vertex $v_i^j$. Clearly, $S_{\min}$ contains all the colors of $P$. Then, assign arbitrary colors from $P$ to the remaining uncolored vertices of $H_S(G)$. Each color in $P$ appears at least twice in each hyperedge: consequently, $H_S(G)$ admits a bipanchromatic $|S_{\min}|$-coloring. From Corollary~\ref{coro1}, $G$ admits $|S_{\min}|$ CIST.
\end{proof}

The statement of Theorem~\ref{thm01} is very strict; indeed, the separator hypergraph $H_S(G)$ of a graph $G$ may not admit a pairwise separated bipanchromatic $k$-coloring, but rather a pairwise separated bipanchromatic $s$-coloring for some $s < k$. It can be hard to check whether such an $s$-coloring exists, but we can use the $k$-bipanchromatic coloring to construct the CIST, even if this coloring is not pairwise separated. To do that, let us introduce the following definition: Let $H$ be a hypergraph and $\phi$ a $k$-bipanchromatic coloring of $H$, and let $P = \{ p_1,p_2,\dots, p_k \}$ denote the set of colors used by $\phi$. A color $p_l \in P$ is said to be a {\em total color} if, for every $i \in \{ 1, \dots, k\} \backslash \{l\} $, there exist two edge-disjoint sets $F_{il}^l,F_{il}^i \subseteq \{ (x,y) | x \in D_l, y \in D_i \}$ such that every vertex of $D_l$ is incident with an edge of $F_{il}^i$, and every vertex of $D_i$ is incident with an edge of $F_{il}^l$. In other words, $D_l$ is pairwise separated from all other colors.

\begin{thm}
Let $G$ be a connected chordal graph. If $H_S(G)$ admits a bipanchromatic $k$-coloring such that $m$ colors are total colors, then $G$ admits $\lfloor \frac{k+m}{2} \rfloor$ CIST.
\label{thm02}
\end{thm}

\begin{proof}
Let $\varphi$ be a bipanchromatic $k$-coloring  of $H_S(G)$ with color classes $D_1,\dots,D_k$. Consider the set of color classes of total colors denoted by $D' = \{ D_1',\dots, D_m' \} \subset D $. According to the definition of total colors, the colors of $D'$ are mutually pairwise separated; moreover, each color of $D'$ is pairwise separated with all non-total colors. Therefore, according to Theorem~\ref{thm01}, $G$ admits $m$ CIST.

Furthermore, consider regrouping the classes of non-total colors into pairs, so that if only one color remains, it is arbitrarily assigned to another class. These $\lfloor\frac{k-m}{2} \rfloor$ new classes of colors are pairwise separated. Indeed, without loss of generality, consider four arbitrary non-total color classes $ D_1,D_2,D_3,D_4$ (or five classes $D_1$, $D_2$, $D_3$, $D_4$, and $D_5$, when only five classes remain, in which case $D_5$ is merged with the class $D_1$ (\textit{i.e.}, $D_1 := D_1 \cup D_5$)); those classes are regrouped into two classes $A = \{ D_1,D_2 \}$ and $B = \{ D_3,D_4 \}$. According to Lemma~\ref{lemme0}, $A$ and $B$ are connected dominating sets. Let 
\[
\left\{
\begin{aligned}
    F_{AB}^A &= \{ (x,y) \in E(G) | (x \in D_1 \wedge y \in D_4) \wedge (x \in D_2 \wedge y\in D_3) \}\\
    F_{AB}^B &= \{ (x,y) \in E(G) | (x \in D_1 \wedge y \in D_3) \wedge (x \in D_2 \wedge y\in D_4) \}
\end{aligned}
\right.
\]
Those two sets are clearly edge-disjoint and pairwise separated with all total colors; moreover, every vertex of $A$ is incident with an edge of $F_{AB}^A$, and every vertex of $B$ is incident with an edge of $F_{AB}^B$. It follows that $G$ admits $m + \lfloor \frac{k-m}{2} \rfloor = \lfloor \frac{k+m}{2} \rfloor$ CIST.
\end{proof}

The number of CIST given by Theorem~\ref{thm02} is not optimal. Figure~\ref{expl1} shows a chordal graph $G$ admitting four CIST. Its separator hypergraph $H_S(G)$ admits the bipanchromatic $4$-coloring displayed on the top right, which has no total colors. Indeed, the bipartite subgraphs between the green and orange classes and between the red and blue classes have edge sets $\{(a,d),(b,d),(d,f),(f,h)\}$ and $\{(c,e),(e,g),(g,i)\}$, respectively. These subgraphs contain only four edges on five vertices and three edges on four vertices, respectively, which is insufficient to obtain two edge-disjoint sets that cover the respective color classes. Hence, no color is total. Therefore, with $k=4$ and $m=0$, Theorem~\ref{thm02} guarantees only two CIST. However, the same hypergraph admits the pairwise-separated bipanchromatic $3$-coloring shown on the bottom right, which guarantees three CIST.
\begin{figure}[ht]
\centering
\begin{tikzpicture}[scale = 0.7]        

        \node[shape=circle,draw=black] (a) at (-5.5,2.5) {$a$};
        \node[shape=circle,draw=black] (b) at (-4,2.5) {$b$}; 
        \node[shape=circle,draw=black] (c) at (-2.5,2.5) {$c$};
        \node[shape=circle,draw=black] (d) at (-1,0) {$d$};
        \node[shape=circle,draw=black] (e) at (-2.5,-2.5) {$e$}; 
        \node[shape=circle,draw=black] (f) at (-4,-2.5) {$f$};
        \node[shape=circle,draw=black] (g) at (-5.5,-2.5) {$g$};
        \node[shape=circle,draw=black] (h) at (-4,-5) {$h$};
        \node[shape=circle,draw=black] (i) at (-5.5,-5) {$i$};    
        \node[shape=circle,draw=black] (j) at (-7,0.8) {$j$};
        \node[shape=circle,draw=black] (k) at (-7,-0.8) {$k$};
        \node[shape=circle,draw=black] (l) at (-7,-3.5) {$l$};    
        \node[shape=circle,draw=black] (m) at (-4.8,-7.5) {$m$};    
        \node[shape=circle,draw=black] (n) at (-2.5,-5) {$n$};    

\path (a) edge (b);
\draw (a) to[out=45,in=135] (c);
\path (a) edge (d);
\path (a) edge (e);
\path (a) edge (f);
\path (a) edge (g);
\path (a) edge (j);
\path (a) edge (k);

\path (b) edge (c);
\path (b) edge (d);
\path (b) edge (e);
\path (b) edge (f);
\path (b) edge (g);
\path (b) edge (j);
\path (b) edge (k);

\path (c) edge (d);
\path (c) edge (e);
\path (c) edge (f);
\path (c) edge (g);
\path (c) edge (h);
\path (c) edge (j);
\path (c) edge (k);
\draw (c) to[out=-25,in=25] (n);

\path (d) edge (e);
\path (d) edge (f);
\path (d) edge (g);
\draw (h) to[out=40,in=-110] (d);
\path (d) edge (j);
\path (d) edge (k);
\draw (m) to[out=0,in=-90] (d);
\path (d) edge (n);

\path (e) edge (f);
\draw (g) to[out=-40,in=-140] (e);
\path (e) edge (h);
\path (e) edge (j);
\path (e) edge (k);
\draw (m) to[out=45,in=-105] (e);
\path (e) edge (n);

\path (f) edge (g);
\path (f) edge (h);
\path (f) edge (i);
\path (f) edge (j);
\path (f) edge (k);
\path (f) edge (l);
\draw (f) to[out=-105,in=90] (m);
\path (f) edge (n);

\path (g) edge (h);
\path (g) edge (i);
\path (g) edge (j);
\path (g) edge (k);
\path (g) edge (l);
\draw (g) to[out=-75,in=90] (m);

\path (g) edge (n);

\path (h) edge (i);
\path (h) edge (l);
\path (h) edge (m);
\path (h) edge (n);

\path (i) edge (l);
\path (i) edge (m);

        \node[shape=circle,draw=black,fill=green] (a) at (1.5,0) {$a$};
        \node[shape=circle,draw=black,fill=green] (b) at (3,0) {$b$}; 
        \node[shape=circle,draw=black,fill=red] (c) at (4.5,0) {$c$};
        \node[shape=circle,draw=black,fill=orange] (d) at (6,0) {$d$};
        \node[shape=circle,draw=black,fill=cyan] (e) at (7.5,0) {$e$}; 
        \node[shape=circle,draw=black,fill=green] (f) at (9,0) {$f$};
        \node[shape=circle,draw=black,fill=red] (g) at (10.5,0) {$g$};
        \node[shape=circle,draw=black,fill = orange] (h) at (12,0) {$h$};
        \node[shape=circle,draw=black,fill=cyan] (i) at (13.5,0) {$i$};

    \draw[] (6,0) ellipse (5.4cm and 1.7cm);
    \draw[] (7.5,0) ellipse (3.6cm and 1.3cm);
    \draw[] (9,0) ellipse (3.8cm and 1.3cm);
    \draw[] (11.2,0) ellipse (3cm and 1.3cm);

        \node[shape=circle,draw=black,fill=green] (a1) at (1.5,-5) {$a$};
        \node[shape=circle,draw=black,fill=green] (b1) at (3,-5) {$b$}; 
        \node[shape=circle,draw=black,fill=red] (c1) at (4.5,-5) {$c$};
        \node[shape=circle,draw=black,fill=orange] (d1) at (6,-5) {$d$};
        \node[shape=circle,draw=black,fill=green] (e1) at (7.5,-5) {$e$}; 
        \node[shape=circle,draw=black,fill=green] (f1) at (9,-5) {$f$};
        \node[shape=circle,draw=black,fill=red] (g1) at (10.5,-5) {$g$};
        \node[shape=circle,draw=black,fill = orange] (h1) at (12,-5) {$h$};
        \node[shape=circle,draw=black,fill=green] (i1) at (13.5,-5) {$i$};

    \draw[] (6,-5) ellipse (5.4cm and 1.7cm);
    \draw[] (7.5,-5) ellipse (3.6cm and 1.3cm);
    \draw[] (9,-5) ellipse (3.8cm and 1.3cm);
    \draw[] (11.2,-5) ellipse (3cm and 1.3cm);

\end{tikzpicture}
\caption{A chordal graph $G$ such that $H_S(G)$ admits a bipanchromatic $4$-coloring (above) and also a pairwise separated bipanchromatic $3$-coloring (below)}
\label{expl1}
\end{figure}
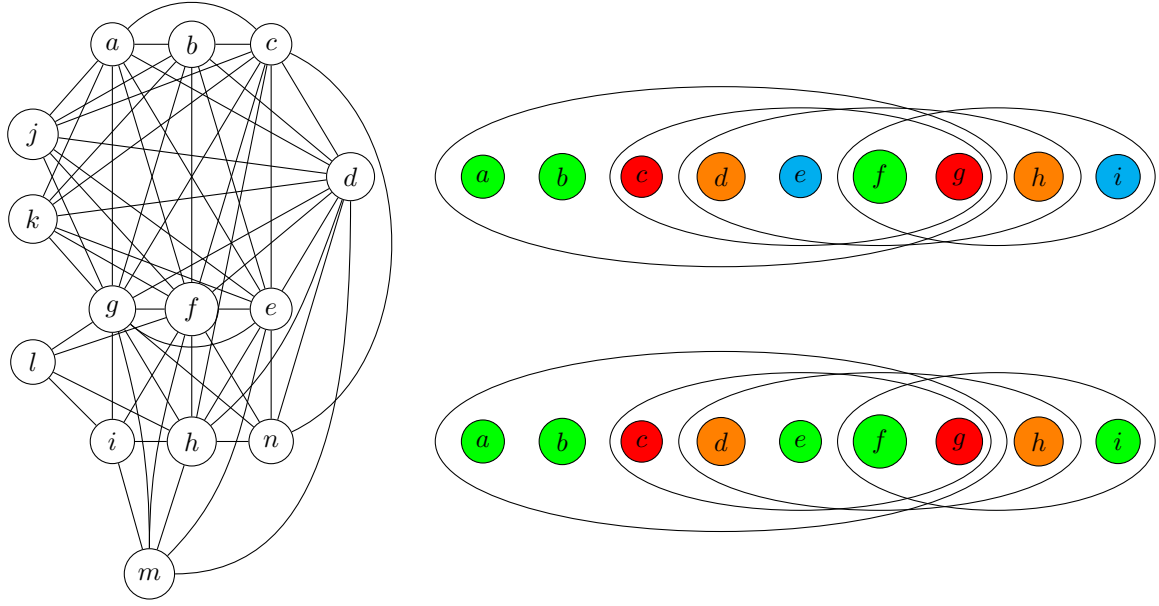

\label{sec4}
\section{Concluding remarks}\label{sec5}
In this paper, we extended the hypergraph approach to CIST from split graphs to chordal graphs. We established a necessary condition for the existence of $k$ CIST, and then proved that a pairwise separated bipanchromatic $k$-coloring of the separator hypergraph guarantees the existence of $k$ CIST. For strictly chordal graphs, bipanchromaticity alone is sufficient and gives explicit lower bounds based on the sizes of the minimal separators. More generally, we obtained a lower bound of $\lfloor \frac{k+m}{2} \rfloor$ CIST from a bipanchromatic $k$-coloring with $m$ total colors, allowing colorings that are not pairwise separated to be used in the construction.

Moreover, the usefulness of bipanchromatic colorings in establishing existence conditions for CIST motivates the study of this coloring problem on its own. Further work could seek sharper bounds on $\chi_p^2(H)$ in terms of $\chi_p(H)$ and $\alpha(H)$, and identify classes of hypergraphs for which the pairing of vertices considered in the conjecture is optimal. Moreover, algorithmic questions could be investigated, including identifying hypergraph classes for which the bipanchromatic number can be computed in polynomial time and constructing approximation algorithms to find bipanchromatic colorings with many colors.

\bibliographystyle{abbrvnat}
\bibliography{reference.bib}

\end{document}